\documentclass[11pt,leqno]{article} 
\usepackage[active]{srcltx}
\usepackage{alphabeta} 

\usepackage{amsfonts,latexsym,amsmath,amssymb,amsthm}
\usepackage{fullpage}
\newtheorem{theorem}{Theorem}[section]
\newtheorem{lemma}[theorem]{Lemma}

\newtheorem{proposition}[theorem]{Proposition}

\newtheorem{remark}[theorem]{Remark}

\theoremstyle{definition}
\newtheorem{definition}[theorem]{Definition}

\theoremstyle{remark}

\newtheorem*{note*}{Note}

\numberwithin{equation}{section}

\makeatletter
\newcommand{\rank}{\mathop{\operator@font rank}}
\newcommand{\conv}{\mathop{\operator@font conv}}
\newcommand{\vol}{\mathop{\operator@font vol}}
\newcommand{\onetagright}{\tagsleft@false}
\makeatother

\newcommand{\Di}{\mathbb{D}}

\renewcommand{\epsilon}{\varepsilon}

\usepackage{color}
\newcommand{\red}[1]{\textcolor{red}{#1}}

\usepackage{hyperref}

\begin{document}
\small

\title{\bf Topological and Algebraic Genericity and Spaceability for an extended chain of sequence spaces}

\medskip

\author{M. Axarlis, I. Deliyanni, Th. Loukidou, V. Nestoridis, K. Papanikos, N. Tziotziou}

\date{}

\maketitle

\centerline {\it Dedicated to the memory of Professor Dimitris Gatzouras}

\begin{abstract}
\noindent \footnotesize We examine topological and algebraic genericity and spaceability for any pair $(X,Y)$, $X\subset Y$, $X\neq Y$ belonging to an extended chain of sequence spaces which contains the $\ell^p$ spaces, $0<p\leq \infty$.
\end{abstract}

\section{Introduction}

In \cite{nestor}, \cite{bern}, the chain of spaces $\cap_{p>\alpha}\ell^p \; (\alpha\geq 0)$, $\ell^p \; (0<p<+\infty)$, $c_0$, $\ell^{\infty}$ was considered and, for any pair $(X,Y)$ with $X,Y,\; X\subsetneqq Y$ belonging to this chain, topological and algebraic  genericity and spaceability were investigated, extending previous results. We recall the definitions. Given a pair of vector spaces $(X,Y)$ as above, we say that we have topological genericity if $X$ is contained in an $F_{\sigma}$ - meager subset of $Y$, equivalently if $Y\setminus X$ is residual in $Y$. This is always the case and then a question that arises naturally is whether $X$ is indeed equal to an $F_{\sigma}$ subset of $Y$ or not. Furthermore, we say that we have algebraic genericity for the pair $(X,Y)$, if there exists a vector subspace $F$ of $Y$, dense in $Y$, such that $F$ is contained in $(Y\setminus X) \cup \{0\}$. Finally, we have spaceability if $(Y\setminus X)\cup \{0\}$ contains a closed infinite dimensional subspace of $Y$.

\noindent In the present paper we extend the above chain by adding the space $A^{\infty}({\mathbb D})$ which is contained in $\cap_{p>0}\ell^p$ and the spaces $H({\mathbb D})\subset {\mathbb C}^{{\mathbb N}_0} $ which contain  $\ell^{\infty}$. The spaces $A^{\infty}({\mathbb D})$ and $H({\mathbb D})$ are spaces of holomorphic functions on the open unit disc $D$ of the complex plane ${\mathbb C}$, but they can also be seen as sequence spaces via the identification of any holomorphic function on ${\mathbb D}$ with the sequence of its Taylor coefficients. More precisely, for $f=\sum\limits_{n=0}^{\infty} a_nz^n$, we have that $f$ belongs to $A^{\infty}({\mathbb D})$ if and only if, for every $k=1,2,\ldots$, it holds that $n^ka_n\rightarrow 0$ as $n\rightarrow +\infty$, while $f$ belongs to $H({\mathbb D})$ if and only if $\lim\sup\sqrt[n]{|a_n|}\leq 1.$

\noindent For $X\subsetneqq Y$ belonging to this extended chain of spaces, we examine topological and algebraic genericity and spaceability, completing thus the results of \cite{nestor} and \cite{bern}. We also mention the remarkable papers \cite{greg} and \cite{pap} which are related to this work. 

\noindent For algebraic genericity and spaceability we refer the reader to \cite{ber pel} and \cite{ar ber pel}. For topological genericity we refer to \cite{bay}.

\noindent This work was completed by the students and the instructors in the frame of an advanced undergraduate course in Analysis at the Department of Mathematics of the National and Kapodistrian University of Athens in Greece.

\section{Topological Genericity}\label{topological}

We begin with the following\\
\begin{proposition}
 Let $\mathbb{C}^{\mathbb{N}_0}$ be the set of sequences $(a_n)_{n=0}^{\infty}$ with $a_n \in \mathbb{C}, n=0,1,2,\dots$, endowed with the usual operations, pointwise addition, scalar multiplication
 $+,\cdot$. Let $X,Y$ be two $F$-spaces which are vector subspaces of $\mathbb{C}^{\mathbb{N}_0}$.
  We assume that convergence of a sequence $a^m= \left (a_n^m \right )_{n=0}^{\infty} $ in either $X$ or $Y$ implies pointwise convergence,
 that is, if $a^m \xrightarrow{m \to \infty} a$ in $X$ or $Y$, then $a^m_n \xrightarrow{m \to \infty} a_n$ for all $n= 0,1,2,\dots$
If $X \subset Y$ then the inclusion map $I:X\rightarrow Y, \, I(\alpha)=\alpha$, is continuous.
\end{proposition}
\begin{proof}
This follows immediately from the closed graph theorem. \\Indeed, let $ (a^m, I(a^m)) =\left (a^m,a^m\right ) \in$ Gr($I$) such that $\left (a^m,a^m\right ) \xrightarrow{m \longrightarrow \infty} \left (a,b\right )$.
It suffices to show that $a=b$. From our assumption, $a^m \rightarrow a$ in $X$. 
It follows that $a_n^m \rightarrow a_n$ as $m \rightarrow \infty$ for every $n$.
Similarly from the convergence in $Y$, we have that $a_n^m \rightarrow b_n$ as $m \rightarrow \infty $  for every $n=0,1,2,\dots$ It follows that $a=b$.
\end{proof}

\begin{proposition}\label{2.2}
If, in addition to the assumptions of Proposition 2.1., $X$ is different from $Y$, then $X$ is included in an $F_{\sigma}$ meager subset of $Y$.
\end{proposition}

\begin{proof}
This follows from Proposition $2.1.$ and a theorem of Banach which is a version of the open mapping theorem (\cite{rud}, Theorem 2.11), since the inclusion map $I:X\rightarrow Y, \, I(a)=a$, is linear, continuous and not surjective.
\end{proof}

The above find application when $X,Y$ are among the spaces $\ell^p \text{ for } 0<p<\infty, \,\, \bigcap_{p>\alpha} \ell^p
\text{ for } 0 \leq \alpha <\infty, \,\, c_0 \,\, \text{and} \,\, \ell^{\infty}$ (\cite{nestor} and \cite{bern}). In the present paper, we extend this chain by adding the spaces $A^{\infty}(\Di) \subset \bigcap_{p>0} \ell^p$ and $\ell^{\infty} \subset H(\Di) \subset \mathbb{C}^{\mathbb{N}_0}$, where $\Di$ is the open unit disc in $\mathbb{C}$. \\ \\
Let us first recall the definitions.

\begin{definition}
Let $H(\Di)$  be the set of all holomorphic functions on the open unit disc $\Di$ and endow this space with the topology of uniform convergence on compact subsets of $\Di$.\\
We consider $H(\Di)$ as a sequence space, by identifying every function $f(z) = \sum_{n=0}^{\infty}a_n z^n$ with the sequence $a=(a_n)_{n=0}^{\infty}$ of its Taylor coefficients. It is well known that $f\in H(\Di) $ if and only if  $ \limsup_n \left \{ \sqrt[n]{|a_n|} \right \} \leq 1$.
\end{definition}

\begin{definition}
Let $A^{\infty}(\Di)$  be the set of holomorphic functions  $f$ on the open unit disc $\Di$ such that $f$ and all its  derivatives $f^{(l)}$ can be continuously extended on the closed unit disc $\overline{\Di}$.\\
We endow $A^{\infty}(\Di)$ with the natural metric $d(f,g)=\sum_{i=0}^{\infty}\frac{1}{2^i} \frac{\|f^{(i)}-g^{(i)} \|_{\infty}}{1+\|f^{(i)}-g^{(i)} \|_{\infty}}$.\\
As before, we identify every $ f \in A^{\infty}(\Di)$ with the sequence $a=(a_n)_{n=0}^{\infty}$ of its Taylor coefficients. It is easy to see that $ f \in A^{\infty}(\Di)$ if and only if $n^k a_n \xrightarrow{n \to \infty}  0$  for every $k$ in $\mathbb{N}_0.$
\end{definition}

\begin{proposition}\label{2.5}
Convergence in $H(\Di)$ implies pointwise convergence.
\end{proposition}
\begin{proof}
Let $f_m(z)= \sum_{n=0}^{\infty}a_n^m z^n$ be a sequence in $H(\Di)$ that converges to $f(z)=\sum_{n=0}^{\infty}a_n z^n$ in $H(\Di)$.\\
It suffices to show that for every $n \in \mathbb{N}_0$ we have 
$a_n^m \xrightarrow{m \rightarrow \infty} a_n$.\\
By the Weierstrass theorem we have that, for every $n \in \mathbb{N}_0$, $f_m^{(n)}$ converges uniformly to $f^{(n)}$ as $m \to \infty$ on each compact subset of $\Di$.
Thus, in particular, for every $n \in \mathbb{N}_0$,
$$a_n^{m} = \frac{f_m^{(n)}(0)}{n!} \xrightarrow{m \rightarrow \infty} \frac{f^{(n)}(0)}{n!} = a_n$$
\end{proof}

\begin{proposition}\label{2.6}
Convergence in  $A^{\infty}(\Di)$ implies pointwise convergence.
\end{proposition}
\begin{proof}
It is obvious that convergence in  $A^{\infty}(\Di)$ implies uniform convergence in $\overline{\Di}$ which 
implies convergence in $H(\Di)$. Thus, by Proposition {\ref{2.5}}, we have pointwise convergence.
\end{proof}

\begin{remark}\label{2.7}
{\rm It is obvious that convergence in either $\ell^{\infty}$ or $\mathbb{C}^{\mathbb{N}_0}$ implies pointwise convergence.}
\end{remark}

\begin{proposition}\label{2.8}
The inclusion $A^{\infty}(\Di) \subset \bigcap_{p>0}\ell^p$ holds, and it is strict.
\end{proposition}

\begin{proof}

 Let $f(z)=\sum_{n=0}^{\infty}a_n z^n \, \in A^{\infty}(\Di)$ , i.e.
$\sum_{n=0}^{\infty} n^k|a_n| < \infty$ for all $k \geq 0$ and let $p >0$. \\
Let $k\in {\mathbb N}$ be such that $kp>1$. \\
We have $n^k |a_n| \to 0 $ so there exists $N>1$ such that $n^k |a_n|<1 $ for every $ n\geq N$. \\
Thus $\sum_{n=0}^{\infty}|a_n|^p \leq \sum_{n=0}^{N-1}|a_n|^p + \sum_{n=N}^{\infty}\left (\frac{1}{n^k}\right )^p< \infty$ since $kp>1$.\\
We now show that the inclusion is strict:\\
Consider the sequence $y=(y_s)$ where
\begin{align*}
    y_s= \left\{ \begin{array}{rcl}
        \sqrt{\frac{1}{s}} & &\mbox{if}\;
         s=2^k \; {\hbox {for   some  k}} \in \mathbb{N} 
        \\ 0 & &\mbox{otherwise} 
        \end{array}\right.
\end{align*}
In other words, $y_{2^k}=\sqrt{\frac{1}{2^k}}$ and $y_s=0$ elsewhere. Then $2^ky_{2^k}=2^{k/2} \rightarrow \infty$
so $y \notin A^{\infty} (D)$.\\
On the other hand, $y \in \bigcap_{p>0} \ell^{p}$ since $\sum_{s=0}^{\infty}|y_s|^p=\sum_{n=1}^{\infty} \left |1 / \sqrt{2} \right |^{np} < \infty$.

\end{proof}

\noindent Next we prove a slightly stronger fact that will be used later.

\begin{remark}\label{rem2.9}
	For every infinite subset $A$ of ${\mathbb{N}_0}$ we can find a sequence  $y \in \bigcap_{p>0}\ell^p \smallsetminus A^{\infty}(\Di)$ which is supported in $A$.
	\end{remark}
\begin{proof}
Let $A=\{l_1, l_2, \dots \}$ where $l_1 < l_2 < \dots$ . We choose $k_1<k_2<\ldots$ such that, for every $n\in{\mathbb N},$  $l_{k_n} \geq 2^n$.\\
We define $y=(y_s)$ by $y_{l_{k_n}}= \sqrt{1/l_{k_n}}$ and $y_s=0$ otherwise.\\
Then $l_{k_n}y_{l_{k_n}}=\sqrt{l_{k_n}} \geq 2^{n/2} \to \infty$ so $y \notin A^{\infty}(\Di)$,
while for every $p>0$ we have:
\begin{align*}
   \sum_{s=0}^{\infty}|y_s|^p \leq  \sum_{n=1}^{\infty}|2^{-{n/2}}|^p < \infty 
\end{align*}
 so $y \in \bigcap_{p>0} \ell^{p}$.
\end{proof}

\begin{proposition}
The inclusion $\ell^{\infty} \subset H(\Di)$ holds and it is strict.
\end{proposition}

\begin{proof}
Let $a=(a_n)_n \in \ell^{\infty}$. Then 
$$\lim\sup\sqrt[n]{|a_n|}\leq \lim\sqrt[n]{||a||_{\infty}}\leq 1.$$
So $a \in H(\Di)$, which implies $\ell^{\infty} \subset H(\Di)$.  \\
Since $\text{limsup} \sqrt[n]{n} = 1$ it follows that the sequence $(n)_n$ is in $H(\Di)$,  but not in $\ell^{\infty}$. \\
So $\ell^{\infty} \varsubsetneq H(\Di)$.
\end{proof}

\begin{proposition}\label{2.11}
The inclusion $H(\Di)\subset \mathbb{C}^{\mathbb{N}_0}$ holds and it is strict.
\end{proposition}

\begin{proof}
It is obvious that $H(\Di) \subset \mathbb{C}^{\mathbb{N}_0}$.  \\
Since $\text{limsup} \sqrt[n]{n^{n+1}} = + \infty $ it follows that $ \left (n^{n+1}\right )_n \notin H(\Di)$. 
 Therefore, $ H(\Di) \varsubsetneq \mathbb{C}^{\mathbb{N}_0}$.
\end{proof}

\begin{theorem}\label{the212}
Consider the chain of spaces $$A^{\infty}(\Di) \varsubsetneq \bigcap_{p>0}\ell ^p \varsubsetneq \ell^a \varsubsetneq \bigcap_{q>a}\ell^q \varsubsetneq \ell^b \varsubsetneq \bigcap_{p>b}\ell^p \varsubsetneq c_0 \varsubsetneq \ell^{\infty} \varsubsetneq H(\Di) \varsubsetneq \mathbb{C}^{{\mathbb{N}}_0}$$
where $a<b$.\\
If $X \varsubsetneq Y$ are two spaces from this chain
then $X$ is contained in an $F_{\sigma}$ meager subset of $Y$.
\end{theorem}
\begin{proof}
This follows by a combination of Propositions \ref{2.2},  \ref{2.5},  \ref{2.6} and \ref{2.7}.  
\end{proof}

\section{A Constructive Approach}\label{constructive}

In the previous section we showed that if $X \varsubsetneq Y$ are spaces as in theorem {\ref{the212}} then X is contained in an $F_{\sigma}$ meager subset of Y. \\
In this section we examine whether X is itself an $F_{\sigma}$ meager subset of Y. Our method will be constructive.
At the same time we  obtain a new proof of theorem {\ref{the212}} without using Banach's theorem.
 \\
\begin{proposition}\label{3.1}
Let $X = A^{\infty}(\Di)$ and $Y$ be a space from the chain of theorem {\ref{the212}} such that $X \varsubsetneq Y$. Then $X$ is an $F_{\sigma \delta}$ subset of $Y$.
\end{proposition} 

\begin{proof}
For $k\in {\mathbb N}_0$, $M\in {\mathbb N},$ let $F_M^k = \left \{ a=(a_n) \in Y  \ | \ n^k|a_n| \leq M \ \forall n\in {\mathbb N}_0 \right \} $. It is clear that $X = A^{\infty}(\Di) = \bigcap_{k=0}^{\infty} \bigcup _{M=1}^{\infty} F_M^k \subset \bigcup_{M=1}^{\infty}F_M^1$ and it remains to show that the sets $F_M^k$ are closed in $Y$. 
 Indeed, fix $k$ and $M$ and let $(a^m)$ be a sequence in $F_M^k$, such that $a^m \xrightarrow{m \to \infty}a$ in $Y$ and thus 
    $a^m_n \xrightarrow{m \to \infty}a_n$ for all $n \in \mathbb{N}_0$. 
    Then for all $n \in \mathbb{N}$ $n^k|a_n^m| \leq M $ and by taking the limit as $m$ goes to $\infty$ we have $ n^k|a_n| \leq M$ which implies that $a \in F_M^k$. This completes the proof.

\end{proof}

\begin{remark}
    {\rm The proof of {\ref{3.1}} can be used to give a new proof of the fact that $A^{\infty}(\Di)$ is contained in an $F_{\sigma}$ meager subset of $Y$. It suffices to show that $\bigcup_{M=1}^{\infty} F_M ^1$ has empty interior in $Y$. Indeed, it is obvious that
   $\bigcup_{M=1}^{\infty}F_M^1$ is a  vector subspace of $Y$. To see that it is a proper subspace, notice that the sequence $y=(y_n)$ of Proposition
     {\ref{2.8}} is in $\bigcap_{p>0}\ell^p \subset Y$ and $y \notin \bigcup_{M=1}^{\infty} F_M^1$ because $(n y_n)$ is not bounded. 
   It follows that $\bigcup_{M=1}^{\infty}F_M^1$ has empty interior in $Y$. \\
    We mention that all cases of Theorem \ref{the212} can be derived by the method of section \red{\ref{constructive}} without using Banach's Theorem. We will not insist on this point.}
    
\end{remark}

\begin{proposition} 
    Let $X = \ell^p$ for $p>0$ and $Y$ be a space from the chain of theorem 2.12 such that $X \varsubsetneq Y$. Then $X$ is an $F_{\sigma}$ meager subset of $Y$.
\end{proposition}

\begin{proof}
It suffices to write $\ell^p = \bigcup_{M=1}^{\infty}\{a = (a_n)_{n=0}^{\infty} \in Y \ | \ \sum_{n=0}^N |a_n|^p \leq M \ \forall N\in {\mathbb N} \} $
as in \cite{nestor}.
Thus,  the set $\ell^p$, being a proper vector subspace of $Y$ has empty interior in $Y$ and is equal to a countable union of closed sets in $Y$.

\end{proof}

\begin{proposition}
  Let $X = \bigcap_{p>c}\ell^p$ for $c\geq0$ and $Y$ be a space from the chain of theorem {\ref{the212}} such that $X \varsubsetneq Y$. Then $X$ is an $F_{\sigma\delta}$ subset of $Y$.
\end{proposition}

\begin{proof}
Let $p_n = c+\frac{1}{n}$.
We have $\bigcap _{p>c}\ell^p = \bigcap_{n=1}^{\infty}\ell^{p_n}$. Since $\ell^p$ is $F_σ$ in $Y$, it follows that $X$ is $F_{σδ}$ in $Y$.
\end{proof}

\begin{remark}
{\rm Obviously $c_0$ is closed in $\ell^{\infty}$.}
\end{remark}

\begin{proposition}
Let $X = c_0$ and $Y=H(\Di)$ or $\mathbb{C}^{\mathbb{N}_0}$. Then $X$ is $F_{\sigma\delta}$ in $Y$.
\end{proposition}

\begin{proof}
$X=\bigcap_{k=1}^{\infty} \bigcup_{n=1}^{\infty}F_n^k$ where $F_n^k= \{a=(a_s) \in Y : |a_s|\leq \frac{1}{k} \ \forall s \geq n \}$.\\
$F_n^k$ are closed in $Y$. Indeed, fix $n,k \in \mathbb{N}$.\\
Let $a^m , \ m=1,2,\dots$ be a sequence in $F_n^k$, such that $a^m \xrightarrow{m \to \infty}a$ in $Y$. According to proposition {\ref{2.5}} and remark {\ref{2.7}} we have 
    $a^m_n \xrightarrow{m \to \infty}a_n$, for all $n \in \mathbb{N}_0$.
Then, for all $s \geq n, \  |a_s^m|\leq \frac{1}{k}$ and by taking the limit as $m$ goes to $\infty$ we have, for all $s \geq n, \  |a_s|\leq \frac{1}{k}$ which implies that $a \in F_n^k$.    
\end{proof}

\begin{proposition} Let $X = \ell^{\infty}$ and $Y$ be a space from the chain of theorem {\ref{the212}} such that $X \varsubsetneq Y$. Then $X$ is an $F_{\sigma}$ subset of $Y$.
\end{proposition}

\begin{proof}
 Let $F_M = \left \{a=(a_n) \ \in Y \ : \ |a_n| \leq M \; {\rm for\; all \;} n\in \mathbb{N}_0 \right \}$. Obviously $X = \bigcup_{M=1}^{\infty}F_M$. We will show that each set $F_M$ is closed in Y.\\
     Indeed, let $a^m$ be a sequence in $F_M$ such that $a^m \xrightarrow{m \rightarrow \infty} a,$ for some $a\in Y$.
     Convergence in $Y$ implies pointwise convergence, that is $a_n^m \xrightarrow{m \rightarrow \infty} a_n$ for every $n \in \mathbb{N}_0$. Since  $|a^m_n|\leq M$ for all $n\in {\mathbb N}_0$ and $m\in {\mathbb N}$,  it follows that $|a_n|\leq M$ for all $n\in {\mathbb N}_0$. Thus, $a \in F_M$.
\end{proof}

\begin{proposition}
      Let $X = H(\Di)$ and $Y = \mathbb{C}^{\mathbb{N}_0}$. Then $X$ is an $F_{\sigma \delta}$ subset of $Y$.
\end{proposition}

\begin{proof}
$X = H(\Di) = \bigcap_{j=1}^{\infty}\bigcup_{k=1}^{\infty} F_k^j $, where $F_k^j = \left \{a= (a_n) \in \mathbb{C}^{\mathbb{N}_0} \ | \ \sqrt[n]{|a_n|} \leq 1+ \frac{1}{j}  \ \forall n\geq k \right \} $.\\
$F_k^j$ are closed in $Y$. Indeed, fix $j,k \in \mathbb{N}$.\\
Let $a^m , \  m=1,2, \dots$ be a sequence in $F_k^j$ such that $a^m \xrightarrow{m \to \infty}a$ in $Y$, so that 
    $a^m_n \xrightarrow{m \to \infty}a_n$ for all $n \in \mathbb{N}_0$.
Then for all $n \geq k, \  \sqrt[n]{|a_n^m|} \leq 1+ \frac{1}{j}  $ and by taking the limit as $m$ goes to $\infty$ we have for all $n \geq k, \   \sqrt[n]{|a_n|} \leq 1+ \frac{1}{j}$ which implies that $a \in F_k^j$.
\end{proof}

\begin{remark}
	{\rm  The proof of Proposition 3.8 gives that $X=H(D)\subseteq \bigcup_{k=1}^{\infty}F_k^1\subseteq Y={\mathbb C}^{{\mathbb N}_0}$ where the set $\bigcup_{k=1}^{\infty}F_k^1$ is an $F_{\sigma}$-meager subset of $Y$. In other words, $Y\setminus X$ contains the complement of $\bigcup_{k=1}^{\infty}F_k^1$ which is a $G_{\delta}$-dense subset of $Y$. We mention that $Y\setminus X$ also contains the set of sequences $(a_n)$ with the property that the power series $\sum_{n=0}^{\infty}a_nz^n$ has 0 radius of convergence or where $\sum_{n=0}^{\infty}a_nz^n$ is a universal power series of Seleznev. It is known that these two last sets are $G_{\delta}$-dense subsets of $Y={\mathbb C}^{{\mathbb N}_0}$ (\cite{bay}). A series $\sum_{n=0}^{\infty}a_nz^n$ is a universal power series of Seleznev if its partial sums approximate uniformly every polynomial on any compact set $K\subset {\mathbb C}\setminus \{0\}$ with connected complement.}
\end{remark}

\begin{remark}
    {\rm In the cases where we show that $X$ is an $F_{σδ}$ in $Y$, we believe that this result can not be improved, that is $X$ is not an $F_σ$ subset of $Y$. This is true in particular in the case where $X=\bigcap_{p>a}\ell_p$ and $Y$ is equal to either $\ell^{b}$ or $\bigcap_{q>b}\ell_q$ for some $0<a<b<\infty$ as shown by Gregoriades in \cite{greg}.}
\end{remark}

\section{Algebraic Genericity} \label{algebraic}

In continuation to the previous project \cite{nestor} we examine whether there is algebraic genericity for a couple of spaces $(X,Y)$, where $X\subsetneq Y$ are spaces belonging to the chain of Theorem \ref{the212}.\\ \\
We recall the  definition:

\begin{definition}
Let $X,Y$ be F-spaces, with $X \subset Y $ and $X \neq Y$. We say that we have algebraic genericity for the couple $(X,Y)$ if there is a vector subspace $F$ of $Y$ dense in $Y$, such that $F \smallsetminus \{0\} \subset Y \smallsetminus X$ 
\end{definition}

\noindent The main result is that if $X$ and $Y$ are two spaces belonging to the chain of theorem 2.12 then we have algebraic genericity for the couple $(X,Y)$. Here we deal with the case $Y\neq \ell^{\infty}$. When $Y = \ell^{\infty}$ the proof,  due to Papathanasiou \cite{pap}, follows a different method since $\ell^{\infty}$ is non separable.

\begin{lemma}\label{lem4.2}
Let $X,Y$ be F-spaces, such that:
\begin{enumerate}
    \item $c_{00}\subset X\subset Y \subset \mathbb{C}^{\mathbb{N}_0}$, $X \neq Y$
    \item If $A\subset \mathbb{N}_0$ is infinite, there exists $y \in Y\smallsetminus X$  supported in $A$.
    \item $c_{00}$ is dense in $Y$ 
    \item For every $a \in X$ and  $A \subset \mathbb{N}_0$ the product $a \chi_A$ belongs to $X$. 
\end{enumerate}
Then we have algebraic genericity for the pair $(X,Y)$
\end{lemma}

\begin{proof}
Since $c_{00}$ is dense in Y it follows that $c_{00}\cap (\mathbb{Q}+i\mathbb{Q})^{\mathbb{N}_0}$ is dense in Y.\\
Let $\{x_j : j\in \mathbb{N}\}$ be an enumeration of $c_{00}\cap (\mathbb{Q}+i\mathbb{Q})^{\mathbb{N}_0}$ and let $(A_j)_{j \in \mathbb{N}}$ be a sequence of pairwise disjoint infinite subsets of $\mathbb{N}$. By condition 2, for every $j\in {\mathbb N}$ there exists $y_j \in Y\smallsetminus X$, $y_j$ supported in $A_j$.\\
Since $Y$ is a topological vector space, for every $j\in {\mathbb N}$, there exists $ c_j \in \mathbb{C} \smallsetminus \{0\}$ such that $c_j y_j \in B_Y(0, \frac{1}{j})$. 
 Let $f_j=x_j+c_j y_j \,\, \text{for every} \,\, j$. From $d_Y(f_j,x_j)<\frac{1}{j}$ and the fact that Y does not have isolated points it follows that $\{f_j: j \in \mathbb{N} \}$ is dense in Y. 
Also, $f_j \notin X$ because $y_j \notin X$. This proves that $F=\left <f_1, f_2, \dots \right >$ is dense in $Y$.\\
It suffices to show that $F \cap X = \{0\}$. \\
Suppose that there exists $\sum_{j=1}^M t_j f_j \in X \smallsetminus \{0\}$, $t_j \in  \mathbb{C}$. Since $x_1,x_2 \dots ,x_M \in c_{00}$ there exist $N$ such that $x_j(n)=0$ for all $j=1,2,\dots,M$ and $n \geq N$.
Let  $j_0 \in \{1,2, \dots, M\}$ be such that $t_{j_0} \neq 0$.\\
Then from assumption 4 we have that $\sum_{j=1}^M t_j f_j \chi_{A_{j_0} \cap [N,\infty)}\in X$
$$\sum_{j=1}^M t_j f_j \chi_{A_{j_0} \cap [N,\infty)}=t_{j_0} y_{j_0} \chi_{A_{j_0} \cap [N,\infty)} = t_{j_0} y_{j_0} \chi_{[N,\infty)} = t_{j_0} y_{j_0} -  t_{j_0} y_{j_0} \chi_{[1,N)} \in X$$
Since $t_{j_0} y_{j_0} \chi_{[1,N)} \in c_{00} \subset X$, $X$ is a vector space and $t_{j_0} \neq 0$ it follows that $y_{j_0} \in X$, which is a contradiction.

\end{proof}

\begin{remark} {\rm Using the terminology of \cite{aron top} (Definition 2.1), the assumptions of our lemma \ref{lem4.2} imply that $Y\setminus X$ is stronger than $c_{00}$. Thus, one can also use Theorem 2.2 of \cite{aron top}  to obtain the result of the previous lemma. We mention that although Theorem 2.2 of \cite{aron top} is stated for Banach spaces, it can easily be generalized to F-spaces.}
	\end{remark}

\begin{proposition}\label{prop4.4}
If $X,Y$ are spaces from the chain of theorem 2.12 such that $X \subsetneq Y$,  and $Y \neq \ell^{\infty}$ then conditions 1,2,3,4 of Lemma \ref{lem4.2} are satisfied.
\end{proposition}

\begin{proof}
Let $X, Y$ be spaces from the {chain of theorem 2.12} such that $X \subset Y$, $X \neq Y$. It is obvious that condition 1 holds.
We now prove that condition 2 holds. Let $X = \ell^p, \bigcap_{p>a}\ell^p, c_0$ or $\ell^{\infty}$.
Since the inclusion $X \subset Y$ is strict, we can choose $a \in Y \smallsetminus X$. Let $A$ be an infinite subset of $\mathbb{N}$. We can spread out the elements $a_n$ in such a way that the support of $a$ is contained in $A$. To be more precise, let $A=\{i_1, \dots, i_k, \dots\}$ be an enumeration of $A$ such that $i_k < i_{k+1}$ for all $k \in \mathbb{N}$. Set:
\begin{align*}
    b_n = \begin{cases}
    a_k,\; n=i_k, \,\, k \in \mathbb{N} \\
    0, \;n \notin A
    \end{cases}
\end{align*}
Then, $y=(b_n)_n \in Y \smallsetminus X$ and has support in $A$.
This proves that condition 2 holds for these spaces.\\
If $X = A^{\infty}(\Di)$ then condition 2 follows from remark {\ref{rem2.9}}.\\ 
If $X = H(\Di)$ then we construct a sequence supported in $A=\{l_1<l_2<l_3 < \cdots\}$ :\\
\begin{align*}
    c_n= \left\{ \begin{array}{rcl}
        n^n & \mbox{if}
        & n=l_k  \text{ for some k}\\ 0 & \mbox{otherwise} 
         \end{array}\right.
\end{align*}
Then, $(c_n)_n \in Y \smallsetminus X$ and has support in $A$.

\vspace{0.5em}
\noindent We now prove that condition 3 holds.\\
If $Y = \ell^p, c_0, \mathbb{C}^{\mathbb{N}_0}$ then it is obvious that $c_{00}$ is dense in $Y$.\\
Let $Y = \bigcap_{p>a}\ell^p$. \\
The fact that the inclusion map among the $\ell^p$ spaces is continuous and $c_{00}$ is dense in each one of the spaces $\ell^p$ allows us to have control over any finite set of $\ell^p$ spaces. This proves that $c_{00}$ is dense in $\bigcap_{p>a}\ell^p$.\\
Let $Y = H(\Di)$. Every $a \in c_{00}$ can be identified with a complex polynomial.
It is well known that every holomorphic $f \in H(\Di)$ can be approached by polynomials, uniformly on the compact subsets of $\Di$.
It follows that $c_{00}$ is dense in $Y$.

\vspace{0.5em} \noindent
We now prove that condition 4 holds.\\
Let $A$ be a subset of $\mathbb{N}_0$ and $a = (a_n)_n \in X$. 
Then $|a_n \chi_{A}| \leq |a_n|$ for all $n \in \mathbb{N}_0$ and from this inequality condition 4 is obvious for the spaces $\ell^p, \bigcap_{q>a}\ell^q, 0\leq a <\infty, c_0 $. 
If $(a_n)_n \in H(\Di)$, equivalently $ \limsup_n \left \{ \sqrt[n]{|a_n|} \right \} \leq 1$, then $\limsup_n \left \{ \sqrt[n]{\left|a_n \chi_{A}\right|} \right \} \leq 1$, which proves that $a\chi_{A} \in H(\Di)$. 
Similarly, if $(a_n)_n \in A^{\infty} (\Di)$, equivalently $n^k a_n \xrightarrow{n \to \infty}  0$  for every $k \in \mathbb{N}$, then $n^k a_n \chi_{A} \xrightarrow{n \to \infty}  0$  for every $k \in \mathbb{N}$, which implies that $a\chi_{A} \in A^{\infty} (\Di)$. 
\end{proof}

\begin{theorem}\label{th4.5} 
If $X,Y$ are spaces from the {chain of Theorem 2.12} with $X \subsetneq Y$ and $Y \neq \ell^{\infty}$, then we have algebraic genericity for the couple $(X,Y)$.
\end{theorem}

\begin{proof}
It follows from Lemma {\ref{lem4.2}} and Proposition {\ref{prop4.4}}.
\end{proof}

\noindent If $Y=\ell^{\infty}$ then $X \subset c_0$. According to {Papathanasiou} \cite{pap} there exists a vector subspace $F$ of $\ell^{\infty}$ dense in $\ell^{\infty}$ such that $F\smallsetminus \{0\} \subset \ell^{\infty}\smallsetminus c_0 \subset \ell^{\infty} \smallsetminus X$. Thus, we have algebraic genericity for the couple $(X, \ell^{\infty})$. Combining this with theorem \red{\ref{th4.5}} we obtain:

\begin{theorem}
Let $(X,Y)$ be spaces from the chain of Theorem 2.12 with $X \subset Y$ and $X \neq Y$. Then we have algebraic genericity for the couple $(X,Y)$.
\end{theorem}

\section{Spaceability}
In the last section we examine whether there is spaceability for a couple of spaces $(X,Y)$, where $X, Y$ are spaces in the  chain of Theorem 2.12. 

\vspace{0.5em} \noindent
Let us first recall the definition:
\begin{definition}
Let $X, Y$ be $F$ spaces, with $X \subset Y$ and $X \neq Y$. We say that we have spaceability for the couple $(X, Y)$ if there exists a closed infinite dimensional subspace $F$ of $Y$ such that $F \smallsetminus \{0\} \subset Y\smallsetminus X$.
\end{definition}

\noindent The main result is that if $X$ and $Y$ are two spaces of the chain of Theorem 2.12 such that $X \subsetneq Y$ then we have spaceability for the couple $(X,Y)$.

\begin{lemma}\label{lem5.2}
Let 
$X, Y$ be F spaces such that:
\begin{enumerate}
    \item $X \subsetneq Y \subset \mathbb{C}^{\mathbb{N}_0}$
    \item If $A\subset \mathbb{N}_0$ is infinite then there exists $y \in Y\smallsetminus X$  supported in $A$.
    \item Convergence in $Y$ implies pointwise convergence.
    \item For every $a \in X$ and  $A \subset \mathbb{N}$ the product $a \chi_A$ belongs to $X$. 
\end{enumerate}
Then we have spaceability for the pair $(X, Y)$.
\end{lemma}

\begin{proof}
Let $(A_j)_{j \in \mathbb{N}}$ be a sequence of pairwise disjoint infinite subsets of $\mathbb{N}$. By condition 2, for every $j$ there exists $y_j \in Y\smallsetminus X$ supported in $A_j$.\\
Consider $F=\overline{\left <y_j | j \in \mathbb{N} \right >}$.\\
It is obvious that $F$ is a closed linear subspace of $Y$. Since the sets $A_j$ are disjoint, it follows that $F$ is infinite dimensional.\\
It remains to show that if $f \in F, \ f \neq 0$ then $ f \notin X$.\\
Indeed, there exists a sequence $f^m \in \left <\{y_j | j \in \mathbb{N} \} \right >$ such that 
$f^m \xrightarrow{m \to \infty} f$ in $Y$ and by condition 3 we have 
$f^m(i) \xrightarrow{m \to \infty} f(i)$ for all $i \in \mathbb{N}_0$.
For every $m$ we can write $f^m=c_1^m y_1+ c_2^m y_2 +c_3^m y_3+ \dots$ where  finitely many of $c_j^m$  are non zero, i.e. for every $m$ the set $\{ j \in \mathbb{N} \ | \ c_j^m \neq 0 \}$ is finite.\\
But $f \neq 0$, so there exists $i_0 \in \mathbb{N}$ such that $f(i_0) \neq 0$.
If $i_0 \notin \bigcup_j A_j$ then $f^m(i_0)=0$ for all $m$, so $f(i_0)=\lim_m f^m(i_0)=0$, which is a contradiction.\\
Hence, $i_0 \in A_{j_0}$ for some $j_0 \in \mathbb{N}$.\\
Since $A_1, A_2, \dots$ are pairwise disjoint, we have $f^m(i)=c_{j_0}^m y_{j_0}(i)$ for all $i \in A_{j_0}$.\\
If $ y_{j_0}(i_0) = 0$ then $f^m(i_0)=0$ for all $m$, which is a contradiction as above, so $ y_{j_0}(i_0) \neq 0$.\\
Let $c_{j_0}= \lim_m c_{j_0}^m =\lim_m \frac{f^m(i_0)}{y_{j_0}(i_0)} = \frac{f(i_0)}{y_{j_0}(i_0)} \neq 0$.\\
Then for all $i \in A_{j_0}$ we have 
$$f(i)= \lim_m f^m(i)= \lim_m c_{j_0}^m y_{j_0}(i)=c_{j_0} y_{j_0}(i)$$
 thus $f \chi_{A_{j_0}}= c_{j_0} y_{j_0} \chi_{A_{j_0}}= c_{j_0} y_{j_0}  \notin X$ and by condition 4 we have 
 $f \notin X$ as needed.

\end{proof}

\begin{theorem}\label{th5.3}
Let $(X,Y)$ be spaces from the chain of Theorem 2.12 with $X \subsetneq Y$. Then we have spaceability for the couple $(X,Y)$.
\end{theorem}

\begin{proof}
It suffices to see that conditions 1-4 of Lemma {\ref{lem5.2}} hold for any pair of spaces $X,Y$  from the {chain of theorem 2.12} with $X \subsetneq Y.$\\
Conditions 2 and 4 have been proved in Proposition 4.4.\\
Condition 3 has been proved in Section 2.
\end{proof}



\bigskip

\bigskip

\footnotesize
\bibliographystyle{amsplain}

\begin{thebibliography}{100}
\footnotesize


\bibitem{ar ber pel}{\rm R.M. Aron, L. Bernal-Gonz\'alez, D. Pellegrino, J.B. Seoane-Sep\'ulveda,} {\it Lineability: The Search for Linearity in Mathematics}, {\rm Monographs and Research Notes in Mathematics, Chapman and Hall/CRC, 2015}
\bibitem{aron top}{\rm R.M. Aron, F.J. Garc\'ia-Pacheco, D. P\'erez-Garc\'ia, J.B. Seoane-Sep\'ulveda,} {\it On dense-lineability of sets of functions on ${\mathbb R}$,} {\rm Topology 48: 149-156, 2009}
\bibitem{bay}{\rm F. Bayart, K.G. Grosse-Erdmann, V. Nestoridis and C. Papadimitropoulos,} {\it Abstract theory of universal series and applications,} {\rm Proc. Lond. Math. Soc. (3), 96(2): 417-463, 2008}
\bibitem{bern}{\rm L. Bernal-Gonz\'alez and V. Nestoridis,} {\it Topological and algebraic genericity in chains of sequence spaces and fuction spaces}, {\rm Bull. Hellenic Math. Soc. 65: 9-16, 2021} 
\bibitem{ber pel}{\rm L. Bernal-Gonz\'alez, D. Pellegrino and J.B. Seoane-Sep\'ulveda,} {\it Linear subsets of nonlinear sets in topological vector spaces,} {\rm Bull. Amer. Math. Soc. (N.S.), 51(1): 71-130, 2014}
\bibitem{greg}{\rm V. Gregoriades,} {\it Intersections of $\ell^p$ spaces in the Borel hierarchy}, {\rm J. Math. Anal. Appl. 498, Issue 1, 124922, 2021}, {\it see also} {\rm arXiv: 2008.12996}
\bibitem{nestor}{\rm V. Nestoridis,} {\it A project about chains of spaces regarding topological and algebraic genericity and spaceability,} {\rm arXiv: 2005.01023, 2020}
\bibitem{pap}{\rm D. Papathanasiou,} {\it Dense lineabity and algebrability of $\ell^{\infty} \setminus c_0$,} {\rm Proc. Amer. Math. Soc. (to appear)}, {\it see also } {\rm arXiv: 2102.03199}
\bibitem{rud}{\rm W. Rudin,} {\it Functional Analysis}, {\rm McGraw-Hill}

\end{thebibliography}

\bigskip

\bigskip

\thanks{\noindent {\bf Keywords and phrases:}  Topological genericity, algebraic genericiy, spaceability, Baire's theorem, $\ell^p$ spaces.}

\smallskip

\thanks{\noindent {\bf 2010 MSC:} Primary 15A03; Secondary 46E10, 46E15.}

\bigskip

\bigskip

\noindent  Department of
Mathematics, University of Athens, Panepistimioupolis 157-84,
Athens, Greece.

\smallskip

\noindent \textit{E-mail Addresses:} \texttt{mike.axa3@gmail.com},
 \texttt{ideliyanni@math.uoa.gr},
 \texttt{inoloukidou@gmail.com},
\texttt{vnestor@math.uoa.gr},
 \texttt{kotsos129@hotmail.com},
\texttt{nataliatz99@gmail.com}

\end{document}